\documentclass[12pt]{article}
\usepackage[all]{xy}
\usepackage{textcomp}
\usepackage{amsmath,amssymb,eucal,eufrak}
\usepackage{graphicx}
\newtheorem{theorem}{Theorem}[section]
\newtheorem{lemma}[theorem]{Lemma}
\newenvironment{proof}{\par{\bf Proof\ }}{}

\newtheorem{definition}[theorem]{Definition}
\newtheorem{example}[theorem]{Example}

\newtheorem{proposition}[theorem]{Proposition}
\newtheorem{corollary}[theorem]{Corollary}
\newtheorem{remark}[theorem]{Remark}

\numberwithin{equation}{section}

 \def\NN{{\mathbb{N}}}

 \def\cH{{\mathcal{H}}} 
  \def\cL{{\mathcal{L}}}
  
  \def\cR{{\mathcal{R}}}
\def\cS{{\mathcal{S}}}

\begin{document}

\title{On generalized inverses and Green's relations}

\author{Xavier Mary\footnote{email: xavier.mary@ensae.fr}\\
\textit{\small Ensae - CREST, 3, avenue Pierre Larousse 92245
Malakoff Cedex, France}}

\date{}

\maketitle

%\subjclass{Primary 15A09; 20M99} \keywords{generalized inverse,
% Green's relations, semigroup}

\begin{abstract}
We study generalized inverses on semigroups by means of Green's
relations. We first define the notion of inverse along an element
and study its properties. Then we show that the classical
generalized inverses (group inverse, Drazin inverse and
Moore-Penrose inverse) belong to this class.
\end{abstract}

%\section{Introduction}
There exist many specific generalized inverses in the literature,
such as the group inverse, the Moore-Penrose inverse \cite{Ben74} or
the Drazin inverse (\cite{Drazin58}, \cite{Ben74}). Necessary and
sufficient conditions for the existence of such inverses are known
(\cite{Green51}, \cite{Drazin58}, \cite{Harte92}, \cite{Huang93},
\cite{Koliha01}, \cite{Koliha02}, \cite{Patricio04},
\cite{Mary2008SM}), as are their properties. If one looks carefully
at these results, it appears that these existence criteria all
involve Green'relations \cite{Green51}, and that all inverses have
double commuting properties. So one may wonder whether we could
unify these different notions of invertible.

We propose here to define a new type of generalized inverse, the
inverse along an element, that is based on Green's relation's
$\cL,\; \cR$ and $\cH$ \cite{Green51}, and the related notion of
trace product (\cite{Clifford56}, \cite{Pastijn82}). It appears that
this notion encompass the classical generalized inverses but is of
richer type.
By deriving general existence criteria and properties of this inverse, we will then recover in a common framework the classical results. The framework is the one of semigroups, hence the
results are directly applicable in rings or algebras.\\

This article is divided as follows: in the first section, we review
the principal definitions and theorems we will use regarding
generalized inverses and Green's relations. In the second section we
define our new generalized inverse, the inverse along an element,
and derive its properties. In the third section we finally show that
the classical generalized inverses belong to this class, and
retrieve their properties.

\section{Preliminaries}

As usual, for a semigroup $S$, $S^1$ denotes the monoid generated by
$S$, and $E(S)$ its set of idempotents. We first review the various
notions of generalized invertibility and then recall Green's
relations, together with some results linked with invertibility.

\subsection*{Generalized inverses}

Basically, a generalized inverse is an element that share some (but not all) the properties of the classical inverse in a group. We review here the classical notions.\\
Let $a\in S$. The element $a$ in $S$ is called regular if $a\in
aAa$, that is there exists $b$ such that $aba=b$ . In this case $b$
is known as an inner inverse of $a$. If there exists $b\in S$,
$bab=b$ then $b$ is an called an outer inverse of $a$. An element
$b$ that is both an inner and an outer inverse is usually simply
called an inverse of $a$. If it verifies only one of the two
conditions, it is called a generalized inverse. The three common
generalized inverses are defined by imposing additional
properties.\\

If $b$ is an inverse (inner and outer) of $a$ that commutes with $a$
then $b$ is a called a commuting inverse (or group inverse) of $a$.
Such an inverse is unique and usually denoted by $a^{\sharp}$. Its
name ``group inverse'' comes from the following result:

\begin{corollary}[(corollary 4 p.~275 in \cite{Clifford56})]\label{cor4Cliff}
If $a$ and $a'$ are mutually inverse elements of $S$ then $aa'=a'a$ if and only if $a$ and $a'$ belong to the same $\cH$-class $H$. If this be the case, $H$ is a group, and $a$ and
$a'$ are inverses therein in the sense of group theory.
\end{corollary}

To study non regular elements, Drazin \cite{Drazin58} introduced a
another a commuting generalized inverse, that is not inner in
general. $a\in S$ is Drazin invertible if there exists $b\in S$ and
$m\in \NN^*$:
\begin{enumerate}
\item $ab=ba$;
\item $a^m=a^{m+1}b$;
\item $b=b^2a$.
\end{enumerate}
%Remark that this inverse is outer but not inner in general, and can
%exists even for non regular elements.
A Drazin inverse of $a$ is
unique if it exists and will be denoted by $a^D$ in
the sequel. %If $m=1$ works in the definition then the Drazin inverse is the group inverse.

Finally, when $S$ is a endowed with an involution $*$ that makes it
an involutive semigroup (or $*$-semigroup), \textit{i.e.} the
involution verifies $(a^*)^*=a$ and $(ab)^*=b^*a^*$, Moore
\cite{Moore20} and Penrose \cite{Penrose55} studied inverses $b$ of
$a$ with the additional property that $(ab)^*=ab$ and $(ba)^*=ba$.
Once again this inverse, if it exists, is unique. It is usually
called the Moore-Penrose inverse (or pseudo-inverse) of $a$ and will
be denoted by
$a^{+}$.% in the sequel.\\

\subsection*{Green's relations}

For elements $a$ and $b$ of $S$, Green's relations $\cL$, $\cR$ and
$\cH$ are defined by:
\begin{enumerate}
\item   $a \cL b$ if and only if $S^1 a = S^1 b$.
\item $a \cR b$ if and only if $aS^1  = bS^1 $.
\item $a \cH b$ if and only if $a\cL b$ and $a\cR b$.
\end{enumerate}
That is, $a$ and $b$ are $\cL$-related (resp. $\cR$-related) if they
generate the same left (resp. right) principal ideal, and
$\cH=\cL\cap \cR$. These are equivalence relations on $S$, and we
denote the $\cL$-class (resp. $\cR$-class, $\cH$-class) of $a$ by
$\cL_a$ (resp. $\cR_a, \cH_a$). The $\cR$ and $\cL$ relations are
dual to one another and left (resp. right) compatible.

In addition to Green's relation it is useful to our purpose to
introduce their generalization $\cR^*$ and $\cL^*$: $a\cR^*b$ if
$a\cR b$ in some oversemigroup of $S$. By lemma 1.1 in
\cite{Fountain82}, this is equivalent to $$(xa=ya \iff xb=yb \;
\forall x,y\in S^1).$$ Relation $\cL^*$ is defined dually. We have
$\cR\subset \cR^*$ and $\cL\subset \cL^*$.

Finally we will also need the notion of trace product: for $a, b\in
S$, we say that $ab$ is a trace product if $ab\in \cR_a\cap \cL_b$.

We will use the following results:

\begin{theorem}[(Theorem 3 p.~277 \cite{Clifford56})]\label{th3Cliff}
Let $a, b\in S$. $ab$ is a trace product if and only if  $\cR_b\cap
\cL_a$ contains an idempotent element; if this be the case then
$$aH_b=H_a b=H_a H_b=H_{ab}=\cR_a\cap \cL_b.$$
\end{theorem}

\begin{lemma}[(lemma 4 p.~272 in \cite{Clifford56})]\label{lem4Cliff}
An idempotent element $e$ of $S$ is a right identity element of
$\cL_e$, a left identity element of $\cR_e$ and a two-sided identity
element of $\cH_e$.
\end{lemma}

\begin{corollary}[(corollary 1 p.~272 and corollary 1 p.~273 in \cite{Clifford56})]\label{cor1Cliff}
$\,$\\ \vspace{-.8cm}
%Let $a$ be a regular element of $S$.
\begin{enumerate}
\item No $\cH$-class contains more than one idempotent element;
\item an $\cH$-class $\cH_b$ contains an inverse of $a\in S$ if and only if both the $\cH$-class $\cR_a \cap \cL_b$ and $\cR_b \cap \cL_a$ contain idempotent elements;
\item no $\cH$-class contains more than one inverse of $a\in S$.
\end{enumerate}
\end{corollary}

To be exact, statement $2.$ is not precisely the same as in
\cite{Clifford56} where it is assumed that $a$ is regular. But this
is a direct consequence of the existence of an idempotent element in
$\cR_a$ by Von Neumann's lemma 6 \cite{VonNeumann36}.

\begin{theorem}[(Theorem 7 p.~169 in \cite{Green51})]\label{th7Green}
$\,$\\ \vspace{-.8cm}
\begin{enumerate}
\item If a $\cH$-class contains an idempotent $e$, then it is a group
with $e$ as the identity element.
\item If for any $a,b\in S$, $a$, $b$ and $ab$ belong to the same $\cH$-class $H$, then $H$ is a group.
\end{enumerate}
\end{theorem}

\section{A new generalized inverse: the inverse along an element}

\begin{theorem}\label{th_equiv}
Let $a,d\in S$. Then four following statements are equivalent:
\begin{enumerate}
\item there exists $b\in S$ verifying $bad=d=dab$ and $b\in dS\cap Sd$;
\item there exists $b\in S$ verifying $b$ is outer inverse of $a$ and $b\cH d$;
\item there exists $b\in S$, there exists an idempotent $e\in \cR_d$, such that $b$ is an inner and outer inverse of $ae$ and $b\cH d$;
\item there exists $b\in S$, there exists an idempotent $f\in \cL_d$, such that $b$ is an inner and outer inverse of $fa$ and $b\cH
d$.
\end{enumerate}
Moreover, if $b\in S$ verifies one of the four statements, it
verifies the four simultaneously.
\end{theorem}

\begin{proof}
$[1. \Rightarrow 2.]$ let $b, x, y\in S$ such that $bad=d=dab$ and
$b=dx=yd$. Then $bab=badx=dx=b$  and $b$ is an outer inverse of $a$.
but $bad=d=dab$ implies that $d\in bS\cap Sb$ and finally
$b\cH d$.\\
$[2. \Rightarrow 3.]$ Let $b\in S$, $bab=b$ and $b\cH d$. Then one
verifies easily that $ba$ is an idempotent in $\cR_d$. First,
$bab=b$ implies $baba=ba$ and $ba$ is idempotent. Second, since
$b\cH d$, there exists $x\in S,\; d=bx$. But $bab=b$ and it follows
that $d=bx=babx$. Also there exists $y\in S,\; dy=b$, and $ba=dya$.
Finally, $b$ is an inner and outer inverse of $a(ba)$ by direct
calculations.\\
$[3. \Rightarrow 4.]$ Suppose $b$ is the generalized inverse of
$ae$, with $e\in \cR_d\cap E(S)$. Since $b\in \cH_d,\; \exists
x,y,x',y'\in S^1,$  $b=dx=x'd,\; d=by=y'b$. Pose $f=xaed$. Then $f$
is idempotent
($ff=xaedxaed=xaebaed=xaed=f$) and $f\cL d$ ($f=x'aed,\; d=by=baeby=dxaeby=df$).\\
Note that by lemma \ref{lem4Cliff}, $bf=b$ and $eb=b$. It follows
that $baeb=bab=bfab=b$. But
$$fabfa=faba=xaedaba=xaey'baeba=xaey'ba=fa$$  and finally $b$ is an inner and outer inverse of
$fa$.\\
$[(4) \Rightarrow (1)]$ Suppose $b$ is the generalized inverse of
$fa$, with $f\in \cL_d\cap E(S)$.  First, $b\cH d$ implies that
$b\in dS\cap Sd$. Also there exists $y,y' \in S,\; d=by=y'b$. But
$bab=b$ ($f$ is a right identity on $\cL_d=\cL_b$ by lemma
\ref{lem4Cliff} and it follows that $bad=baby=by=d=y'b=y'bab=dab$.
This ends the proof. \vskip-1em
\end{proof}

\begin{definition}\label{df1} Let $a,d\in S$. We say that $b\in S$ is an inverse of $a$ \textbf{along d}
if it verifies one of the four equivalent statements of theorem
\ref{th_equiv}. If moreover the inverse $b$ of $a$ along $d$
verifies $aba=a$, we say that $b$ is an \textbf{inner} inverse of
$a$ along $d$.
\end{definition}

%As a direct consequence of the definition, we get:
%\begin{corollary}\label{lemdd'}
%Let $a,d,d'\in S$.
%\begin{enumerate}
%\item Suppose $d\cH d'$. Then $a$ is invertible along $d$ if and only if it is invertible along $d'$ and in this case, any inverse of along $d$ is an inverse along
%$d'$.
%\item Conversely, suppose there exists $b$ inverse of $a$ along $d$ and along $d'$. Then $d\cH d'$.
%\end{enumerate}
%\end{corollary}

Note that we have also proved that $ba$ and $ab$ are idempotents in
the $\cR$ and $\cL$-class of $d$ respectively:
\begin{corollary}\label{propouter}
$ba\in \cR_d\cap E(S)$ and $ab\in \cL_d\cap E(S)$.
\end{corollary}
%\begin{proof}
% \vskip-1em
%\end{proof}

\begin{example}\label{exT}
Let $S=T_3$ be the full transformation semigroup, that consists of all functions from the set $\{1, 2, 3\}$ to itself with multiplication the composition of functions. We write $(a b
c)$ for
the function which sends $1$ to $a$, $2$ to $b$, and $3$ to $c$.\\
The egg-box diagram form $T_3$ is as follows ($\cR$-classes are rows, $\cL$-classes columns and $\cH$-classes are
squares; bold elements are idempotents):\\

$\,$ \begin{tabular}{cccccccc}
  \cline{1-3}
  \multicolumn{1}{|c|}{\textbf{(1 1 1)}} & \multicolumn{1}{|c|}{\textbf{(2 2 2)}} & \multicolumn{1}{|c|}{\textbf{(3 3 3)}} &  &  &  & &  \\
  \cline{1-3} \cline{4-6}
   &  &  & \multicolumn{1}{|c|}{\textbf{(1 2 2)},}  &\multicolumn{1}{|c|}{\textbf{(1 3 3)},}  &\multicolumn{1}{|c|}{(2 3 3),}  & & \\
   &  &  & \multicolumn{1}{|c|}{(2 1 1)} & \multicolumn{1}{|c|}{(3 1 1)}  & \multicolumn{1}{|c|}{(3 2 2)} & & \\
   \cline{4-6}
   &  &  & \multicolumn{1}{|c|}{(2 1 2),}  &\multicolumn{1}{|c|}{(3 1 3),}  &\multicolumn{1}{|c|}{\textbf{(3 2 3)},}  & &  \\
   &  &  & \multicolumn{1}{|c|}{\textbf{(1 2 1)}} & \multicolumn{1}{|c|}{(1 3 1)}  & \multicolumn{1}{|c|}{(2 3 2)} &  & \\
   \cline{4-6}
   &  &  & \multicolumn{1}{|c|}{(2 2 1),}  &\multicolumn{1}{|c|}{(3 3 1),}  &\multicolumn{1}{|c|}{(3 3 2),}  & & \\
   &  &  & \multicolumn{1}{|c|}{(1 1 2)} & \multicolumn{1}{|c|}{\textbf{(1 1 3)}}  & \multicolumn{1}{|c|}{\textbf{(2 2 3)}} & & \\
   \cline{4-8}
   &  &  &  &  &  & \multicolumn{2}{|c|}{\textbf{(1 2 3)}, (2 3 1),} \\
   &  &  &  &  &  & \multicolumn{2}{|c|}{(3 1 2), (1 3 2),} \\
   &  &  &  &  &  & \multicolumn{2}{|c|}{(3 2 1), (2 1 3)} \\
  \cline{7-8}\\
\end{tabular}

For instance $\cH_{(232)}=\{(232),(323)\}$. Direct computations give
that $a=(221)$ is (inner) invertible along $d=(232)$ with inverse
$b=(323)$, and that $a'=(122)$ and $a''=(123)$ are invertible along
$d=(232)$ but not inner invertible (the inverse is $b=(323)$).
$(111)$ is not invertible along $(232)$.
\end{example}

We now state our second theorem, that deals with existence and
uniqueness of inverses along an element:
\begin{theorem}\label{thmexun}
Let $a,d\in S$.
\begin{enumerate}
\item $a$ is invertible along $d$ if and only if there exists $e\in \left(\cR_d\cap E\left(S\right)\right)$, $\left(ae\right)d$ and $d\left(ae\right)$ are trace
products.
%\item $a$ is invertible along $d$ if and only if $ad\cL d$, $da\cR
%d$ and $\cH_{ad}, _; \cH_{da}$ are groups.
\item If an inverse along $d$ exists, it is unique.
\end{enumerate}
\end{theorem}

\begin{proof}
$\,$\\ \vspace{-.8cm}
\begin{enumerate}
\item[\underline{\textsc{Existence:}}]
Suppose $a$ is invertible along $d$ and let $e\in \cR_d\cap E\left(S\right)$ be the associated idempotent. Then by corollary \ref{cor1Cliff} the $\cH$-classes $\cR_{ae}\cap \cL_d$ and
$\cL_{ae}\cap \cR_d$ contain idempotents elements and by theorem \ref{th3Cliff}
$\left(ae\right)d$ and $d\left(ae\right)$ are trace products.\\
Suppose now there exists $e\in\left(\cR_d\cap E\left(S\right)\right)$, $\left(ae\right)d$ and $d\left(ae\right)$ are trace products. Then by corollary \ref{cor1Cliff} $ae$ admits a
generalized inverse in $\cH_d$ and $a$ is invertible along $d$.
\item[\underline{\textsc{Uniqueness:}}] Let $b$ and $c$ be two
generalized inverses of $a$ along $d$. Then $c=(ba)c=b(ac)b$ by
lemma \ref{lem4Cliff} and lemma \ref{propouter}.
\end{enumerate}
\vskip-1em
\end{proof}

The uniqueness of the inverse along an element allows us to
introduce the following notations: if it exists, we note $a^{\angle
d}$ the inverse of $a$ along $d$ and $a^{\measuredangle d}$ if it is
a inner inverse ($aba=a$).

\begin{corollary}\label{corinner}
Let $a,d\in S$. $a$ is inner invertible along $d$ if and only if $ad$ and $da$ are trace products.
\end{corollary}

\begin{proof}
Suppose $a$ is inner invertible along $d$ with inverse $b$. Pose
$e=ba$. Then $e\in \cR_d\cap E(S)$ by corollary \ref{propouter}, and
by theorem \ref{thmexun} $\left(ae\right)d$ and
$d\left(ae\right)$ are trace products. But $ae=aba=a$ since $b$ is an inner inverse of $a$ and $ad$ and $da$ are trace products.\\
Conversely, if $ad$ and $da$ are trace products, then corollary \ref{cor1Cliff} gives the desired result. \vskip-1em
\end{proof}

\begin{example}
Let once again $S=T_3$ be the full transformation semigroup of
example \ref{exT}. Then the inner invertibility of $a=(221)$ along
$d=(232)$ follows from the existence of the idempotent $e=(121)$ in
the $\cH$-class $\cR_d\cap \cL_a$ and of the idempotent $f=(223)$ in
the $\cH$-class $\cR_a\cap \cL_d$.
\end{example}

\begin{example}\label{exM}
Let $\cS$ be the subsemigroup of $\mathcal{M}_3(\NN)$ generated by
the matrices $$a=\left(
           \begin{array}{ccc}
            1 &0 & 0 \\
            0 & 1 & 0 \\
            0 & 0 & 0 \\
           \end{array}
         \right) \; b=\left(
           \begin{array}{ccc}
            0 &1 & 1 \\
            1 & 0 & 0 \\
            0 & 0 & 0 \\
           \end{array}
         \right)\;c=\left(
           \begin{array}{ccc}
            0 &1 & 0 \\
            1 & 0 & 0 \\
            0 & 0 & 0 \\
           \end{array}
         \right)\;d=\left(
           \begin{array}{ccc}
            1 &0 & 0 \\
            0 & 1 & 1 \\
            0 & 0 & 0 \\
           \end{array}
         \right)$$
         Then $a\cR b\cR c\cR d$ (the semigroup is right simple), $a\cL c$ and $b\cL d$. Since $a$
         and $d$ are idempotents, each $\cL$-class contains
         idempotents elements and it follows that any product of two
         elements is a trace product. Finally any element is
         invertible along another one, or equivalently any
         $\cH$-class $\{a,c\}$, $\{b,d\}$ is an inverse transversal
         (see note 3.16 in \cite{Zhang2002}).
\end{example}

As a corollary, we get an interesting characterization of the
inverse of $a$ along $d$ in terms of the group inverse
$(ad)^{\sharp}$ (or $(da)^{\sharp}$), as well as a second existence
criteria:
\begin{corollary}\label{corgroup}
Let $a,d\in S$. The three following statements are equivalent:
\begin{enumerate} \item $a$ is invertible along $d$;
\item $ad\cL d$ and $\cH_{ad}$ is a group;
\item $da\cR d$ and $\cH_{da}$ is a group.
\end{enumerate}
In this case $a^{\angle d}=d(ad)^{\sharp}=(da)^{\sharp}d$
\end{corollary}

\begin{proof}
$\,$\\ %\vspace{-.4cm}
$1.\Rightarrow 2.$ First, suppose $a$ is invertible along $d$ and
let $e$ be the idempotent in $\cR_d$ such that $b$ is the inverse of
$ae$ in $\cH_d$. By theorem \ref{thmexun}, $(ae)d$ is a trace
product and it follows that $ae\cH_d=\cH_{aed}$ or equivalently
since $e$ is a right identity on $\cH_d$, $a\cH_d=\cH_{ad}$. Since
$b\in \cH_d$ and $ab$ is idempotent, the $\cH$-class $\cH_{ad}$
contains the idempotent $ab$ hence it is a group. Finally equality
$bad=d$ gives $ad\cL d$.

$2. \Rightarrow 3.$ Suppose now $ad\cL d$ and $\cH_{ad}$ is a group,
and let $x\in S$ such that $d=xad$. From
$ad=ad(ad)^{\sharp}ad=(ad)(ad)(ad)^{\sharp}=(ad)^{\sharp}(ad)(ad)$
we get $d=xad=x(ad)(ad)(ad)^{\sharp}=(da)d(ad)^{\sharp}$ and $da\cR
d$. To prove that $\cH_{da}$ is a group, by theorem \ref{th7Green}
we only need to prove that $da\cH (da)^2$. But
$$da=xada=x(ad)(ad)(ad)^{\sharp}a=x(ad)(ad)(ad)^{\sharp}(ad)(ad)^{\sharp}a=d(ad)(ad)(ad)^{\sharp}(ad)^{\sharp}a$$
and $da\cR (da)^2$. But since $d=dad(ad)^{\sharp}$ we have also
$$da=dad(ad)^{\sharp}a=d(ad)^{\sharp}(ad)(ad)(ad)^{\sharp}a=d(ad)^{\sharp}(ad)^{\sharp}(ad)(ad)a$$
and $da\cL (da)^2$. Finally $da\cH (da)^2$.

$3. \Rightarrow 1.$ Suppose now $da\cR d$ and $\cH_{da}$ is a group.
Pose $b=(da)^{\sharp}d$ and let $x\in S$ such that $d=dax$. Then
$bad=(da)^{\sharp}dad=(da)^{\sharp}dadax=dax=d$, and
$dab=da(da)^{\sharp}d=da(da)^{\sharp}dax=dax=d$.\\
But also $b=(da)^{\sharp}d=(da)^{\sharp}dax=da(da)^{\sharp}x$ and
$b\in dS\cap Sd$. Finally $b=(da)^{\sharp}d$ is the inverse of $a$
along $d$. \vskip-1em
%
%Finally let $(ad)^{\sharp}$ be the inverse of $ad$ in $\cH_{ad}$. It
%follows that $(ad)(ad)^{\sharp}$ is the unique idempotent of the
%group $\cH_{ad}$. But so is $ab$ and we get $ab=ad(ad)^{\sharp}$.
%Finally, equalities $bab=b$ and $bad=d$ give
%$$b=bab=bad(ad)^{\sharp}=d(ad)^{\sharp}.$$ The same proof works for
%$da$.
\end{proof}

\begin{example}
Let $\cS$ be the subsemigroup of $\mathcal{M}_3(\NN)$ of example
\ref{exM}. The inverse of $b$ along $c$ is $b^{angle
c}=c(bc)^{\sharp}=ca^{\sharp}=ca=c$.
\end{example}

We finally prove an interesting result regarding commutativity. If
$A$ is a subset of the semigroup $S$, $A'$ denotes as usual the
commutant of $A$ and $A''$ its bicommutant.

\begin{theorem}\label{commut}
Let $a,d\in S$ and pose $A=(a,d)$. If $a$ is invertible along $d$,
then $a^{\angle d}\in A''$.
\end{theorem}

\begin{proof}
Let $b$ be the inverse of $a$ along $d$. It then verifies
$bad=d=dab$ and $b\in dS\cap Sd$. Let $x,y \in S$ such that
$b=dx=yd$. Suppose $c\in A'$. Then
\begin{eqnarray*}
cd&=&cbad=cdab=dacb\\ %badc=bcad\\
&=& dc=badc=bcad=dabc
\end{eqnarray*}
hence $cbad=bcad,\quad dabc=dacb$. Then \begin{eqnarray*}
cb&=&cbab=cbadx=bcadx=bcab\\ %badc=bcad\\
&=& bacb=ydacb=ydabc=babc=bc
\end{eqnarray*} and $b\in
A''$. \vskip-1em \end{proof}

\begin{remark} If $da=ad$, the two previous results give that $b=a^{\angle
d}$ commutes with $a$ and $d$ and that $\cH_d=\cH_{ad}$ is a group.
\end{remark}

\section{Inverses along $d$ and classical inverses}

One of main interest of this notion of inverse along an element is
that the classical generalized inverses belong to this class:

\begin{theorem}\label{thclassical}
Let $a\in S$. ($S$ is a $*$-semigroup in $3.$)  \begin{enumerate}
\item $a$ is group invertible if and only if it is invertible along $a$. In this case the inverse along $a$ is inner and coincide with the group inverse.
\item $a$ admits a Drazin inverse if and only if it is invertible along some $a^m,\; m\in \NN$, and in this case the two coincide.
\item $a$ is Moore-Penrose invertible if and only if it is invertible along $a^*$. In this case the inverse along $a^*$ is inner and coincide with the Moore-Penrose inverse.
\end{enumerate}
\end{theorem}

\begin{proof}
$\,$
\begin{enumerate}
\item[\underline{Group inverse:}] Suppose $a$ is group invertible. Then $\cH_a$ is
a group that contains the group inverse $a^{\sharp}$. It follows
that $a$ is invertible along $a$,
with inverse $a^{\measuredangle a}=a^{\sharp}$. \\
Conversely, if $a^{\angle a}$ exists, then by corollary
\ref{corgroup} $\cH_a=\cH_{a^2}$ is a group. $a$ is then group
invertible with inverse $a^{\sharp}\in H_a$ and by uniqueness of the
inverse along $a$, $a^{\angle a}=a^{\sharp}=a^{\measuredangle a}$.
\item[\underline{Drazin inverse:}]
Suppose $a$ Drazin invertible (with Drazin inverse $a^D$). Then (Theorem 7 in \cite{Drazin58}) there exists $m\in \NN^*$, there exists $e$ idempotent in $\cH_{a^m}$, $ae=ea\in
\cH_{a^m}$. Then $\cH_{a^m}$ is a group (theorem \ref{th7Green}) and $a$ is invertible along $a^m$ (take $e$ for idempotent). Moreover (see the proof of Theorem 7 in \cite{Drazin58}),
the inverse of $ae$ in $\cH_{a^m}$ is precisely
the Drazin inverse, hence $a^{D}=a^{\angle a^m}$.\\
Conversely, suppose there exists $m\in \NN^*$, $a$ invertible along $a^m$. Then by corollary \ref{corgroup} $a^{m+1}\cH a^m$. But (Theorem 4 p~510 in \cite{Drazin58}) $a$ is Drazin
invertible if and only if it is strongly $\pi$-regular, \textit{i.e.} there exists $m\in \NN^*$, $x,y\in S^1$,
$$a^{m+1}x=a^m=ya^{m+1}$$ or using Green's relation if and only if
there exists $m\in \NN^*, a^{m+1}\cH a^m$. Finally $a$ is Drazin
invertible and by the previous result, the two inverses coincide.
\item[\underline{M-P inverse:}] Suppose $a$ Moore-Penrose invertible with Moore-Penrose inverse $a^+$. Then
$$a^{+}=(a^{+}a)a^{+}=\left(a^{+}a\right)^{*}a^{+}=a^{*}\left(a^{+}\right)^{*}a^{+}$$
$$a^{+}=a^{+}\left(aa^{+}\right)=a^{+}\left(aa^{+}\right)^{*}=a^{+}\left(a^{+}\right)^{*}a^{*}$$
$$a^{*}=\left(aa^{+}a\right)^{*}=\left(a^{+}a\right)^{*}a^{*}=a^{+}a a^{*}$$
$$a^{*}=\left(aa^{+}a\right)^{*}=a^{*}\left(aa^{+}\right)^{*}=a^{*}aa^{+}$$
These four relations imply that $a^{+}\cH a^*$, hence $a$ is inner invertible along $a^*$ with $a^{\measuredangle a^{*}}=a^{+}$.\\
Conversely, suppose $a$ is invertible along $a^*$ with inverse $b$. Let $e\in \cR_{a^{*}}\cap E(S)$. Then $\left(ae\right)a^*=aa^*$ is a trace product and $aa^* \cL a^*$, and also by
transposition $aa^* \cR a$. Finally $aa^*$ is a trace product.\\
But working with $f$ idempotent in $\cL_{a^{*}}$ give that $a^*\left(fa\right)=a^*a$ is a trace product and $a^*a \cR a^*$, and also by transposition $a^*a \cL a$. $a^*a$ is then also
a trace product and the inverse is inner.\\
We finally verify that $ab$ and $ba$ are hermitian using corollary
$\ref{corgroup}$:
$ab=aa^*(aa^*)^{\sharp}=(aa^*)^{\sharp}aa^*=b^*a^*=(ab)^*$ and
$ba=(a^*a)^{\sharp}(a^*a)=(a^* a)(a^*a)^{\sharp}=a^*b^*=(ba)^*$.
\end{enumerate}
\vskip-1em
\end{proof}

Combining theorem \ref{thmexun} (or corollary \ref{corgroup}) and theorem \ref{thclassical}, we then get directly the following %(known and less known)
existence criteria and commuting relations for the classical
inverses \cite{Green51}, \cite{Drazin58}, \cite{Harte92},
\cite{Huang93}, \cite{Koliha01}, \cite{Koliha02}, \cite{Patricio04},
\cite{Mary2008SM}:

\begin{corollary}\label{corex}
Let $a\in S$. \begin{enumerate}
\item A group inverse $a^{\sharp}$ exists if and only if
$a^2Ha$, and $a^{\sharp}\in \left(a\right)''$.
\item A Drazin inverse $a^D$ exists if and only if there exists $m\in \NN^*, a^{m+1}\cH
a^m$, and $a^{D}\in \left(a\right)''$.
\item A Moore-Penrose inverse $a^{+}$ exists if and only if $aa^*\in \cR_a$ %\cap \cL_{a^*}$ and  $a^*a\in\cR_{a^*}\cap \cL_{a}$
and $a^*a\in \cL_{a}$, and $a^{+}\in \left(a,a^*\right)''$.
\end{enumerate}
\end{corollary}

Note that many other results involving classical inverses are then straightforward consequences of theorem \ref{thclassical}. We give two instances of this:

%Take for instance proposition 2 in \cite{Patricio04}:
\begin{proposition}[(proposition 2 p.~162 in \cite{Patricio04})]
Given a in a ring $R$ with involution $*$, the following conditions hold:
\begin{enumerate}
\item If $aR = a^*R$ then $a^{+}$ exists with respect to $*$ iff $a^{\sharp}$ exists, in which case $a^{+}=a^{\sharp}$.
\item If $a^{+}$ exists with respect to $*$, $a^{\sharp}$ exists and $a^{+}=a^{\sharp}$ then $aR = a^*R$.
\end{enumerate}
\end{proposition}

\begin{proof}
By transposition, $(aR = a^*R)\iff (Ra^*=Ra) \iff (a\cH a^*)$. Since
the inverse along an element $d$ depends only on the $\cH$-class of
$d$, theorem \ref{thclassical} then give the desired result.
\vskip-1em
\end{proof}

Remark that in rings with involution, $(aR = a^*R)\iff (aa^*=a^*a)$, and we could have used corollary \ref{corgroup} instead.

\begin{theorem}[(Theorem 5.3 p.~144 in \cite{Koliha02})]
Let $R$ be a ring with involution $*$. An element $a\in R$ is Moore-Penrose invertible if and only if $aa^*\cR^* a$, $a^* a\cL^* a$ and $a^*a$ is group invertible. Then also $aa^*$ is
group invertible and
$$a^{+}=(a^*a)^{\sharp}a^*=a^*(aa^*)^{\sharp}$$
\end{theorem}

\begin{proof}
Suppose $a^{+}$ exists. Then $a$ is invertible along $a^*$. By corollary \ref{corgroup}, $\cH_{aa^*}$ and $\cH_{a^*a}$ are groups and
$$a^{\angle d}=(a^*a)^{\sharp}a^*=a^*(aa^*)^{\sharp}$$
Also by corollary \ref{corgroup} $aa^*\in \cR_a$ and $a^*a\in \cL_{a}$. $\cR\subset\cR^*$ and $\cL\subset\cL^*$, hence $aa^*\cR^* a$, $a^* a\cL^* a$.\\
Conversely, suppose $aa^*\cR^* a$, $a^* a\cL^* a$ and $a^*a$ is
group invertible. Then $\cH_{a^*a}$ is a group hence it contains a
idempotent $f=xa^* a$ that verifies $f\cL^*a$. . From $ff=f$ we get
$af=a$, and finally  $a\cL f\cL a^* a$. The conclusion then follows
from corollary \ref{corgroup} and theorem \ref{thclassical}.
\vskip-1em\end{proof}

\bibliographystyle{amsplain}
%\bibliography{Bibliototale}

\end{document}